\documentclass[12pt,reqno]{amsart}
\usepackage{fullpage}
\usepackage{times}
\usepackage{amsmath,amssymb,amsthm,url}
\usepackage[utf8]{inputenc}
\usepackage[english]{babel}
\usepackage{comment}
\usepackage{bbm}
\usepackage{enumerate}
\usepackage{bm}
\usepackage{graphicx}
\usepackage{mathrsfs}
\usepackage[colorlinks=true, pdfstartview=FitH, linkcolor=blue, citecolor=blue, urlcolor=blue]{hyperref}
\usepackage{tabstackengine}
\stackMath

\newtheorem{thm}{Theorem}[section]
\newtheorem{lem}[thm]{Lemma}
\newtheorem{prop}[thm]{Proposition}
\newtheorem{cor}[thm]{Corollary}

\theoremstyle{definition}
\newtheorem{defn}[thm]{Definition}

\newtheorem{rem}[thm]{Remark}

\newcommand{\R}{\mathbb{R}} 
 
\newcommand{\N}{\mathbb{N}}
\newcommand{\Z}{\mathbb{Z}}
\newcommand{\F}{\mathbb{F}}

\title{Restricted sumsets in multiplicative subgroups}
\author{Chi Hoi Yip}
\address{School of Mathematics\\ Georgia Institute of Technology\\ Atlanta, GA 30332\\ United States}
\email{cyip30@gatech.edu}
\begin{document}

\begin{abstract}
We establish the restricted sumset analogue of the celebrated conjecture of S\'{a}rk\"{o}zy on additive decompositions of the set of nonzero squares over a finite field. More precisely, we show that if $q>13$ is an odd prime power, then the set of nonzero squares in $\F_q$ cannot be written as a restricted sumset $A \hat{+} A$, extending a result of Shkredov. More generally, we study restricted sumsets in multiplicative subgroups over finite fields as well as restricted sumsets in perfect powers (over integers) motivated by a question of Erd\H{o}s and Moser. We also prove an analogue of van Lint-MacWilliams' conjecture for restricted sumsets, which appears to be the first analogue of Erd{\H{o}}s-{K}o-{R}ado theorem in a family of Cayley sum graphs.
\end{abstract}

\maketitle

\section{Introduction}

Throughout the paper, let $p$ be an odd prime and $q$ a power of $p$. Let $\F_q$ be the finite field with $q$ elements. Let $d \mid (q-1)$ such that $d>1$. We denote $S_d(\F_q)=\{x^d: x \in \F_q^*\}$ to be the subgroup of $\F_q^*$ with order $\frac{q-1}{d}$.  If $q$ is assumed to be fixed, for brevity, we simply write $S_d$ instead of $S_d(\F_q)$.  

A celebrated conjecture of S\'{a}rk\"{o}zy \cite{S12} asserts that if $p$ is a sufficiently large prime, then $S_2$ does not admit a {\em nontrivial additive decomposition}, that is, $S_2$ cannot be written as $S_2=A+B$, where $A, B \subset \F_p$ and $|A|, |B| \geq 2$. Recall that the {\em sumset} $A+B$ is defined to be $A+B=\{a+b: a \in A, b \in B\}$.  This conjecture is still widely open. Recently, Hanson and Petridis \cite{HP} made important progress on this conjecture: they showed that S\'{a}rk\"{o}zy's conjecture holds for almost all primes. More generally, one can study the additive decomposition problem for any multiplicative subgroup of a finite field; see for example \cite{HP, S16, S20, S13, Y24}.

We establish the restricted sumset analogue of S\'{a}rk\"{o}zy's conjecture below, namely that the set of nonzero squares in a finite field does not admit a restricted sumset decomposition. Recall that for a given set $A$, its \emph{restricted sumset} is given by $A \hat{+} A=\{a+b: a,b \in A, a \neq b\}$. 

\begin{thm} \label{thm:main}
If $q>13$ is an odd prime power, then the set of nonzero squares in $\F_q$ cannot be written as a restricted sumset of a set. In other words, $S_2 \neq A\hat{+}A$ for any $A \subset \F_q$.
\end{thm}

As one of the fundamental objects, restricted sumsets appear frequently in the study of additive number theory and additive combinatorics. For example, they appear in the celebrated Erd\H{o}s–Heilbronn conjecture (an analogue of Cauchy-Devanport theorem for restricted sumsets), first confirmed by Dias da Silva and Hamidoune \cite{DH94} (see also \cite{ANR96} for a different proof by Alon, Nathanson, and Ruzsa). They also naturally appear in a question of Erd\H{o}s \cite{E63} and Moser \cite{M64} (independently) related to perfect squares, as well as the study of cliques in Cayley sum graphs \cite{G05, GM16}. In particular, we will also discuss some implications of our results on these two types of problems.

Our main motivation for establishing Theorem~\ref{thm:main} is to extend a result due to Shkredov \cite{S14}, where he showed that if $p>13$ is a prime, then $S_2 \subset \F_p$ cannot be written as $A \hat{+}A$. His proof is very delicate~\footnote{On the other hand, it is easy to show if $q \geq 5$ is an odd prime power and $A \subset \F_q$, then $S_2 \neq A+A$; a simple proof can be found in \cite[Theorem 3.2]{S14} as well as \cite[Remark 4.3]{Y24}. We refer to \cite{Y24} for a general discussion of expressing a multiplicative subgroup as $A+A$.}: it is based on Fourier analytic methods and it also uses some results from the theory of perfect difference sets. He~\footnote{private communication} remarked that it is challenging to extend his proof over a general finite field with prime power order, mainly due to a corresponding result in the theory of perfect difference sets seems to be absent. He also remarked that his proof cannot be extended to study the same problem for multiplicative subgroups with index at least $3$ (that is, $S_d$ with $d \geq 3$). Indeed, our proof of Theorem~\ref{thm:main} is completely different from his proof. Moreover, our techniques can be used to handle general multiplicative subgroups with some additional assumptions. In particular, we prove the following result, which is of a probabilistic flavor. 

\begin{thm}\label{thm:density}
Let $s$ be a positive integer and $d \geq 3$. If $s$ is even, further assume that $d$ is not twice a perfect square. Among the set of primes $p \equiv 1 \pmod d$, the lower asymptotic density of primes $p$ such that there is no $A \subset \F_{p^{s}}$ with $A\hat{+}A=S_d(\F_{p^{s}})$ is at least $1-\frac{(d-1)^r}{4d^r}$, where $r=\lceil s/2 \rceil -1$.
\end{thm}

In particular, when $s=1$ (which corresponds to prime fields), the above lower asymptotic density is at least $\frac{3}{4}$. Also observe that as $s \to \infty$, the above lower asymptotic density tends to $1$. When $s$ is even and $d$ is twice a perfect square, the following theorem supplements Theorem~\ref{thm:density}.

\begin{thm}\label{thm:square}
Let $k \geq 3$ be a positive integer and let $d=2k^2$. If $q$ is an even power of a prime $p \equiv 1 \pmod d$, then $A \hat{+} A \neq S_d$ for any $A \subset \F_q$.
\end{thm}

To prove the above results, one key ingredient we develop is the following theorem. Roughly speaking, if $A\hat{+}A=S_d$, then $A$ is ``very close to" being a Sidon set.

\begin{thm}\label{thm: Sidon}
Let $d \geq 2$ and let $p \equiv 1 \pmod d$ be a prime. Let $q$ be a power of $p$ such that $\frac{q-1}{d} \geq 3$. Assume that there is $A \subset \F_q$ such that $A\hat{+}A=S_d$. If $|A|$ is odd or $0 \in A$, then $A$ is a Sidon set with $\{2a: a \in A\} \cap S_d=\emptyset$, and $$q=\frac{d|A|(|A|-1)}{2}+1.$$  
If $|A|$ is even, then 
$$
|A|=2\bigg\lceil \sqrt{\frac{q-1}{2d}} \bigg \rceil, \quad \text{and} \quad \sqrt{\frac{q-1}{2d}} \in \bigg(\frac{1}{2},\frac{3}{4}\bigg) \pmod 1.
$$
\end{thm}

Recall $A=\{a_1, a_2, \ldots, a_N\} \subset \F_q$ is a \emph{Sidon set} if all pairwise sums $a_i+a_j$ (for $i \leq j$) are different. Theorem~\ref{thm: Sidon} is partially inspired by an interesting connection between S\'{a}rk\"{o}zy's conjecture and co-Sidon sets described in the recent work of Hanson and Petridis \cite{HP} (see also a related work by Lev and Sonn \cite{LS17}). It is straightforward to use  Theorem~\ref{thm: Sidon} to deduce Theorem~\ref{thm:square}. Theorem~\ref{thm: Sidon} also implies the following corollary.

\begin{cor}\label{cor:Siegel}
Let $d \geq 2$ and $k \geq 3$. If $k$ is even, additionally assume that $d \neq 8$. Then there exists $p_0=p_0(d,k)$, such that if $p>p_0$ is a prime such that $p \equiv 1 \pmod d$ and $A \subset \F_{p^k}$ such that $A\hat{+}A=S_d(\F_{p^k})$, then $|A|$ is even and $0 \not \in A$.
\end{cor}

Next we discuss upper bounds on $|A|$ assuming $A \hat{+} A \subset S_d \cup \{0\}$. Using a standard double character sum estimate, one can show that $|A|< \sqrt{q}+3$ (see Lemma~\ref{lem: trivial}). When $q=p$ is a prime, we show that this square root upper bound can be improved to roughly $\sqrt{2q/d}$, which provides a multiplicative factor $\sqrt{2/d}$ improvement.

\begin{thm}\label{thm:ub}
Let $d \geq 2$ and let $p \equiv 1 \pmod d$ be a prime. Assume that $A \subset \F_p$.  If $A \hat{+} A \subset S_d$, then $|A|\leq \sqrt{2(p-1)/d+1}+1$; if $A \hat{+} A \subset S_d \cup \{0\}$, then $|A|\leq \sqrt{2(p-1)/d+1}+2$. 
\end{thm}

When $q$ is square and $d \geq 3$, we show that this upper bound from character sum estimates can be improved to $\sqrt{q}$ below. Moreover, this is in general best possible since the $\sqrt{q}$ bound can be achieved by an infinite family.

\begin{thm}\label{thm:VLM}
Let $d \geq 3$. Let $q \equiv 1 \pmod d$ be an odd prime power and a square. If $A \subset \F_q$ such that $A\hat{+}A \subset S_d \cup \{0\}$, then $|A| \leq \sqrt{q}$, with the equality holds if and only if $d \mid (\sqrt{q}+1)$ and $A=\alpha \F_{\sqrt{q}}$, where $\alpha \in S_d$.
\end{thm}

A well-known conjecture due to van Lint and MacWilliams \cite{vLM78} states that if $q$ is an odd prime power and a square, and $A$ is a subset of $\F_{q}$ such that $0,1 \in A$, $|A|=\sqrt{q}$, and $A-A \subset S_2 \cup \{0\}$, then $A$ must be given by the subfield $\F_{\sqrt{q}}$. The conjecture was first proved by Blokhuis \cite{Blo84}, and its various generalizations were confirmed in \cite{AY22, Szi99, Y22}. Thus, Theorem~\ref{thm:VLM} establishes an analogue of van Lint-MacWilliams' conjecture for restricted sumsets. Theorem~\ref{thm:ub} and Theorem~\ref{thm:VLM} can be also formulated in terms of clique number and maximum cliques in the corresponding Cayley sum graphs; we refer to Section~\ref{subsec:Cayleysum} for more discussions. We also refer to Remark~\ref{rem:EKR} on a connection between Theorem~\ref{thm:VLM} and the celebrated Erd{\H{o}}s-{K}o-{R}ado theorem \cite{EKR}.

Finally, we discuss the implications of our results to a related problem over integers. Erd\H{o}s \cite{E63} and Moser \cite{M64} independently asked whether for all $k$ there are integers $a_1<\ldots<a_k$ such that $a_i+a_j$ is a perfect square for all $1 \leq i <j \leq k$. It is easy to verify that the Uniformity Conjecture \cite{CHM} implies that the answer to their question is negative (see the related discussion in \cite[Section 5]{SS21}). Lagrange \cite{L81} and Nicolas \cite{N77} found a set of $6$ integers such that the sum of any two of them is a perfect square (see also a recent paper of Choudhry \cite{C15}), and it is unknown whether there is a set of $7$ integers satisfying such property. 

In the other direction, Rivat, S\'{a}rk\"{o}zy, and Stewart \cite[Theorem 6]{RSS99} proved that if $A \subset\{1, \ldots, N\}$ and $A \hat{+} A$ is contained in the set of perfect squares, then $|A|\leq (36+o(1))\log N$. Gyarmati \cite[Theorem 9]{G01} considered a generalization of the problem by Erd\H{o}s and Moser, and proved that if $d \geq 2$ is fixed, and $A,B \subset \{1,2,\ldots, N\}$ such that $A+B$ is contained in the set of perfect $d$-th powers, then $\min(|A|, |B|)\leq (4d+o(1)) \log N$. In particular, her result implies that the upper bound in the result of Rivat, S\'{a}rk\"{o}zy, and Stewart can be improved from $(36+o(1))\log N$ to $(16+o(1))\log N$. We further improve their upper bound to $(1+o(1))\log N$, as a special case of the following more general result:
\begin{thm}\label{thm:integers}
Let $d \geq 2$. If $A\subset\{1,2,\ldots, N\}$ such that $A \hat{+} A$ is contained in the set of perfect $d$-th powers, then as $N \to \infty$, we have 
$$
|A| \leq \frac{(2+o(1))\phi(d)}{d}\log N.
$$
\end{thm}

\subsection*{Outline of the paper}
In Section~\ref{sec2}, we provide extra background on the topic. In Section~\ref{sec3}, we estimate $|A|$ assuming $A\hat{+}A \subset S_d$. In particular, we prove Theorem~\ref{thm:ub} and Theorem~\ref{thm:VLM}. The goal of Section~\ref{sec4} is to investigate the possibility of decomposing a multiplicative subgroup as a restricted sumset. In particular, we prove Theorem~\ref{thm:main}, Theorem~\ref{thm:density}, and Theorem~\ref{thm:square}. To achieve that, we need to first establish Theorem~\ref{thm: Sidon} on the connection between this problem with Sidon sets. Finally, in Section~\ref{sec5}, we discuss the applications of our main results to the question of Erd\H{o}s and Moser and its variant, and we prove Theorem~\ref{thm:integers}.

\section{Background}\label{sec2}

\subsection{Hyper-derivatives}

Our proof is based on Stepanov's method, inspired by \cite{HP, S69, Y24}.

We recall a few basic properties of hyper-derivatives (also known as Hasse derivatives); we refer to a general discussion in \cite[Section 6.4]{LN97}. 

\begin{defn}
Let $c_0,c_1, \ldots c_d \in \F_q$. If $n$ is a non-negative integer, then the $n$-th hyper-derivative of $f(x)=\sum_{j=0}^d c_j x^j$ is
$$
E^{(n)}(f) =\sum_{j=0}^d \binom{j}{n} c_j x^{j-n},
$$
where we follow the standard convention that $\binom{j}{n}=0$ for $j<n$, so that the $n$-th hyper-derivative is a polynomial.
\end{defn}

\begin{lem}[{\cite[Corollary 6.48]{LN97}}]\label{lem:differentiate}
Let $n,d$ be positive integers. If $c \in \F_q$, then 
$E^{(n)}\big((x+c)^d\big)=\binom{d}{n} (x+c)^{d-n}.$
\end{lem}

\begin{lem}[{\cite[Lemma 6.51]{LN97}}]\label{lem:multiplicity}
Let $f$ be a nonzero polynomial in $\F_q[x]$. If $c$ is a root of $E^{(n)}(f)$ for $n=0,1,\ldots, m-1$, then $c$ is a root of multiplicity at least $m$. 
\end{lem}

\begin{lem} [Leibniz rule for hyper-derivatives, {\cite[Lemma 6.47]{LN97}}]\label{Leibniz}
If $f, g \in \F_q[x]$, then
$$
E^{(n)}(fg)= \sum_{k=0}^n E^{(k)} (f) E^{(n-k)} (g).
$$
\end{lem}

\subsection{Square-root upper bound}

\begin{lem}\label{lem: trivial}
Let $A \subset \F_q$. If $A+A \subset S_d \cup \{0\}$, then $|A|\leq \sqrt{q}$, with equality holding only if $q$ is a square and $A=-A$. If $A\hat{+}A \subset S_d \cup \{0\}$, then $|A|< \sqrt{q}+3$.    
\end{lem}
\begin{proof}
Let $\chi$ be a multiplicative character of $\F_q$ with order $d$. We have the following double character sum estimate (see for example \cite[Theorem 2.6]{Y22}):
 \begin{equation}\label{eq:double}
 \bigg|\sum_{a,b\in A}\chi(a+b)\bigg|  \leq \sqrt{q}|A|\bigg(1-\frac{|A|}{q}\bigg).
 \end{equation}

We first assume that $A+A \subset S_d \cup \{0\}$. Note that for each $a \in A$, we have $a+b \in S_d$ for each $b \in A$, unless $a+b=0$. It follows that
 \begin{equation}\label{eq:lb}
\sum_{a,b\in A}\chi(a+b)=|A|^2-\#\{(a,b) \in A \times A: a+b=0\}\geq |A|^2-|A|,
 \end{equation}
and the equality holds if and only if $A=-A$. Combining inequality~\eqref{eq:double} and inequality~\eqref{eq:lb}, we have
$$
|A|^2-|A| \leq \sqrt{q}|A|\bigg(1-\frac{|A|}{q}\bigg).
$$
It follows that $|A|\leq \sqrt{q}$, and the equality holds only if $A=-A$.

Next, we work under the weaker assumption that $A\hat{+}A \subset S_d \cup \{0\}$. The proof is essentially the same, and the only difference is $\chi(2a)$ could be anything. We instead have
$$
|A|^2-3|A| \leq \bigg|\sum_{a,b\in A}\chi(a+b)\bigg|\leq \sqrt{q}|A|\bigg(1-\frac{|A|}{q}\bigg)<\sqrt{q}|A|,
$$
which implies that $|A|<\sqrt{q}+3$.
\end{proof}

\subsection{Sumsets in multiplicative subgroups}\label{sec: sumset}

The following theorem is an extension of \cite[Theorem 1.2]{HP}, which is the main result in \cite{HP} by Hanson and Petridis, to all finite fields, with an extra assumption on the non-vanishing of a binomial coefficient. Its proof is based on Stepanov's method.

\begin{thm}[{\cite[Theorem 1.1]{Y24}}]\label{thm:sumset}
Let $d \mid (q-1)$ such that $d>1$. If $A,B \subset \F_q$ such that $A+B\subset S_d \cup \{0\}$ and $
\binom{|B|-1+\frac{q-1}{d}}{\frac{q-1}{d}}\not \equiv 0 \pmod p,$ 
then 
$$|A||B|\leq \frac{q-1}{d}+|A \cap (-B)|.$$
\end{thm}

As remarked in \cite{Y24}, in general, the condition on the binomial coefficient cannot be dropped. On the other hand, when $q$ is a prime, it is easy to verify that the condition on the binomial coefficient always holds, and thus Theorem~\ref{thm:main} recovers \cite[Theorem 1.2]{HP}.  Theorem~\ref{thm:sumset} can be used to make progress on S\'{a}rk\"{o}zy's conjecture and its generalization; see \cite{HP, Y24}. 

We remark that Theorem~\ref{thm:sumset} can already be used to deduce non-trivial results on restricted sumsets. For example, it is straightforward to apply Theorem~\ref{thm:sumset} to show the following: if $A \subset \F_p$ and $A \hat{+} A \subset S_d$, then $|A|\leq 2\sqrt{p/d}+O(1)$; indeed, we can write $A$ as the disjoint union of two sets $B$ and $C$ such that $|B|$ and $|C|$ differ by $1$, and then $B+C \subset A \hat{+} A \subset S_d$. However, such ``naive" applications of Theorem~\ref{thm:sumset} are not strong enough for our applications to restricted sumset decompositions. Furthermore, under the same setting, Theorem~\ref{thm:ub} provides a bound of the form $\sqrt{2p/d}+O(1)$, which is much better. For our applications, we still require Theorem~\ref{thm:sumset}, but we will use it in a rather indirect way. Moreover, we also need to establish an analogue of Theorem~\ref{thm:sumset} that is particularly designed for restricted sumsets, which we discuss next.

\section{Upper bounds on $|A|$ assuming $A\hat{+}A \subset S_d$}\label{sec3}

In this section, our main aim is to give an upper bound on $|A|$ provided that $A\hat{+}A \subset S_d$. Note that if $A+A \subset S_d \cup \{0\}$, then we can use Theorem~\ref{thm:sumset} to give an upper bound on $|A|$. Thus we can further assume that $A+A \not \subset S_d \cup \{0\}$, and indeed we will take advantage of this additional assumption in a crucial way in the following proofs. Since the restricted subset $A \hat{+} A$ is generally much more difficult to analyze compared to the sumset $A+A$, when we apply Stepanov's method in this setting, the arguments we present will be more delicate compared to those used to study sumsets \cite{HP, Y24}. 

\subsection{Applications of Stepanov's method}

The following proposition may be viewed as an analogue of Theorem~\ref{thm:sumset} in the restricted sumset setting.

\begin{prop}\label{prop:main}
Let $d \geq 2$ and let $q \equiv 1 \pmod d$ be an odd prime power. Let $A \subset \F_q$ with $|A|=N$ such that $A\hat{+}A \subset S_d$ while $A+A \not \subset S_d \cup \{0\}$. 
\begin{enumerate}
    \item If $N$ is odd, then
    $$
    \frac{N(N-1)}{2}+\#\{a \in A: 2a \in S_d \cup \{0\}\}  \leq \frac{q-1}{d}.
    $$
    \item  If $N$ is even, then
    $$
    \frac{(N-1)^2+1}{2} \leq \frac{q-1}{d}.
    $$
    Moreover, further assume that $0 \in A$, then we have the stronger estimate $$\frac{N(N-1)}{2}+\#\{a \in A: 2a \in S_d\}  \leq \frac{q-1}{d}.$$
\end{enumerate}
\end{prop}

\begin{proof}
Let $A=\{a_1,a_2,\ldots, a_N\}$. If $N=1$, the statements are trivially true. Next, assume $N \geq 2$. If $0 \in A$, without loss of generality we assume that $a_N=0$. Let $n=N$ if $n$ is odd, and let $n=N-1$ if $n$ is even. Let $m=(n-1)/2$; then $m \geq 1$.

Since $A \hat{+} A \subset S_d$, we have 
\begin{equation}\label{eq:obs}
(a_i+a_j)^\frac{q-1}{d} (a_j-a_i)=a_j-a_i    
\end{equation}
for each $1 \leq i,j \leq N$ (note that when $i=j$, both sides are equal to $0$). This simple observation will be used repeatedly in the following computation.

The Vandermonde matrix $$(a_i^j)_{1 \leq i \leq n, 0 \leq j \leq n-1}$$ is invertible. Let $c_1,c_2,...,c_n$ be the unique solution of the following system of equations:
\begin{equation} \label{system1} 
\left\{
\TABbinary\tabbedCenterstack[l]{
\sum_{i=1}^n c_i a_i^j=0,  \quad 0 \leq j \leq 2m-1=n-2\\\\
\sum_{i=1}^n c_i a_i^{n-1}=1
}\right.    
\end{equation}
We note that $c_i \neq 0$ for each $1 \leq i \leq n$; otherwise we must have $c_1=c_2=\ldots=c_n=0$ in view of the first $n-1$ equations in system~\eqref{system1}, which contradicts the last equation in system~\eqref{system1}.

Consider the following auxiliary polynomial 
\begin{equation}\label{poly1}
f(x)=-(-1)^m+\sum_{i=1}^n c_i (x+a_i)^{m+\frac{q-1}{d}} (x-a_i)^{m}\in \F_q[x].    
\end{equation}
First, observe that the degree of $f$ is at most $\frac{q-1}{d}$. Indeed, for each $0 \leq j \leq 2m-1$, the coefficient of $x^{2m+\frac{q-1}{d}-j}$ in $f(x)$ is 
\begin{align*}
 &\sum_{i=1}^n  \sum_{k=0}^j \bigg(\binom{m+\frac{q-1}{d}}{k} c_i a_i^{k} \cdot \binom{m}{j-k}  (-a_i)^{j-k}\bigg)\\
 &=\sum_{k=0}^j \binom{m+\frac{q-1}{d}}{k} \binom{m}{j-k} \cdot \bigg(\sum_{i=1}^n c_i a_i^{k} (-a_i)^{j-k}\bigg)\\
 &=\bigg(\sum_{i=1}^n c_i a_i^{j} \bigg) \cdot \bigg(\sum_{k=0}^j \binom{m+\frac{q-1}{d}}{k} \binom{m}{j-k} (-1)^{j-k}\bigg)=0
\end{align*}
by the assumption in system~\eqref{system1}.

For each $1\leq j \leq N$, equation~\eqref{eq:obs} and system~\eqref{system1} imply that
\begin{align*}
E^{(0)} f (a_j)
= f (a_j)
&= -(-1)^m+\sum_{i=1}^n c_i (a_j+a_i)^{m+\frac{q-1}{d}} (a_j-a_i)^{m} \\
&=  -(-1)^m+\sum_{i=1}^n c_i (a_j+a_i)^{m} (a_j-a_i)^{m}\\
&=  -(-1)^m+\sum_{i=1}^n c_i (a_j^2-a_i^2)^{m}    \\
&= -(-1)^m+ \sum_{\ell=0}^m \binom{m}{\ell} a_j^{2(m-\ell)} (-1)^\ell\bigg(\sum_{i=1}^n c_i a_i^{2\ell}\bigg)   \\
&= -(-1)^m+ (-1)^m \bigg(\sum_{i=1}^n c_i a_i^{n-1}\bigg)   
=0.
\end{align*}

Next, we compute the hyper-derivatives of $f$ on $A$. We first prove the following claim: for each $1 \leq j \leq N$, if $0 \leq k_1\leq m$ and $0\leq k_2<m$ such that $1\leq k_1+k_2 \leq m$, then we have 
\begin{equation}\label{eq:claim}
\sum_{i=1}^n c_i E^{(k_1)}[(x+a_i)^{m+\frac{q-1}{d}}](a_j) \cdot E^{(k_2)}[(x-a_i)^m](a_j)=0.
\end{equation}
Indeed, by Lemma~\ref{lem:differentiate} and equation~\eqref{eq:obs}, we have
\begin{align*}
&\sum_{i=1}^n c_i E^{(k_1)}[(x+a_i)^{m+\frac{q-1}{d}}](a_j) \cdot E^{(k_2)}[(x-a_i)^m](a_j)\\
&= \binom{m+\frac{q-1}{d}}{k_1} \binom{m}{k_2} \bigg(\sum_{i=1}^n c_i (a_j+a_i)^{m-k_1+\frac{q-1}{d}} (a_j-a_i)^{m-k_2}\bigg) \\
&= \binom{m+\frac{q-1}{d}}{k_1} \binom{m}{k_2} \bigg(\sum_{i=1}^n c_i (a_j+a_i)^{m-k_1} (a_j-a_i)^{m-k_2}\bigg) \\
&= \binom{m+\frac{q-1}{d}}{k_1} \binom{m}{k_2} \cdot\sum_{\ell_1=0}^{m-k_1} \sum_{\ell_2=0}^{m-k_2} \binom{m-k_1}{\ell_1} \binom{m-k_2}{\ell_2} \bigg(\sum_{i=1}^n c_i a_j^{m-k_1-\ell_1} a_i^{\ell_1}  a_j^{m-k_2-\ell_2}(-a_i)^{\ell_2}\bigg)\\
&= \binom{m+\frac{q-1}{d}}{k_1} \binom{m}{k_2} \cdot\sum_{\ell_1=0}^{m-k_1} \sum_{\ell_2=0}^{m-k_2} \binom{m-k_1}{\ell_1} \binom{m-k_2}{\ell_2} a_j^{(m-k_1-\ell_1)+(m-k_2-\ell_2)} (-1)^{\ell_2}\bigg(\sum_{i=1}^n c_i  a_i^{\ell_1+\ell_2} \bigg)\\
&=0,
\end{align*}
where we again use the assumptions in system~\eqref{system1}, since we always have $\ell_1+\ell_2\leq 2m-k_1-k_2 \leq 2m-1$ for the exponent in the last summand.

From the above claim~\eqref{eq:claim} and Lemma~\ref{Leibniz}, it follows that for each $1\leq j \leq N$ and $1 \leq r \leq m-1$, we have  
\begin{align*}
E^{(r)} f (a_j)
=  \sum_{i=1}^n c_i \bigg(\sum_{k=0}^r E^{(k)}[(x+a_i)^{m+\frac{q-1}{d}}](a_j) \cdot E^{(r-k)}[(x-a_i)^{m}](a_j)\bigg) =0.
\end{align*}
Similarly, for each $1 \leq j \leq N$, we have
\begin{align}
E^{(m)} f (a_j)
&=  \sum_{i=1}^n c_i \bigg(\sum_{k=0}^m E^{(k)}[(x+a_i)^{m+\frac{q-1}{d}}](a_j) \cdot E^{(m-k)}[(x-a_i)^{m}](a_j)\bigg) \notag \\
&= \sum_{i=1}^n c_i E^{(0)}[(x+a_i)^{m+\frac{q-1}{d}}](a_j) \cdot E^{(m)}[(x-a_i)^{m}](a_j) \notag\\
&= \sum_{i=1}^n c_i (a_j+a_i)^{m+\frac{q-1}{d}}. \label{Neven}
\end{align}

Recall that if $j \neq i$, then $a_j+a_i \in S_d$ and thus $(a_j+a_i)^{\frac{q-1}{d}}=1$. If $1 \leq j \leq n$, then equation~\eqref{Neven} further simplifies to
\begin{align}
E^{(m)} f (a_j)
&= c_j (2a_j)^{m+\frac{q-1}{d}}+\sum_{\substack{1\leq i \leq n \\i \neq j}} c_i (a_j+a_i)^{m} \notag\\
&= c_j (2a_j)^{m+\frac{q-1}{d}} -c_j (2a_j)^m +\sum_{i=1}^n c_i (a_j+a_i)^{m} \notag\\
&= c_j (2a_j)^m \bigg((2a_j)^{\frac{q-1}{d}}-1\bigg)+\sum_{k=0}^m \binom{m}{k} a_j^{m-k} \bigg(\sum_{i=1}^n c_i a_i^k\bigg) \notag\\
&=c_j (2a_j)^m \bigg((2a_j)^{\frac{q-1}{d}}-1\bigg) \label{m+1}.
\end{align}
Note that if $2a_j=0$, that is, $a_j=0$, then $E^{(m)} f (a_j)=0$. If $2a_j \in S_d$, then $(2a_j)^{\frac{q-1}{d}}=1$ and we also have $E^{(m)} f (a_j)=0$. Conversely, if $E^{(m)} f (a_j)=0$, then we must also have $2a_j \in S_d \cup \{0\}$ since $c_j \neq 0$. Thus, $E^{(m)} f (a_j)=0$ if and only if $2a_j \in S_d \cup \{0\}$. Since $A +A \not \subset S_d \cup \{0\}$, there is $1\leq j_0\leq n$ such that $2a_{j_0} \not \in S_d \cup \{0\}$, and thus $E^{(m)} f (a_{j_0})\neq 0$, which implies that $f$ is not identically $0$. 

In view of Lemma~\ref{lem:multiplicity}, for each $1 \leq j \leq n$, we have shown that $a_j$ is a root of $f$ with multiplicity at least $m$; moreover, if additionally $2a_j \in S_d \cup \{0\}$, then $a_j$ is a root with multiplicity at least $m+1$. To complete the proof, we consider the parity of $N$.

(1) $N$ is odd. In this case, $n=N$ and $m=\frac{N-1}{2}$. We conclude that
$$
\frac{N(N-1)}{2}+\#\{a \in A: 2a \in S_d \cup \{0\}\}=mN+\#\{1 \leq j \leq N: 2a_j \in S_d \cup \{0\}\} \leq \deg f \leq \frac{q-1}{d}.
$$

(2) $N$ is even. In this case, $n=N-1$. Note that $a_i+a_N \in S_d$ for each $1 \leq i \leq n$, and thus equation~\eqref{Neven} and system~\eqref{system1} imply that
\begin{align*}
E^{(m)} f (a_N)= \sum_{i=1}^n c_i (a_N+a_i)^{m+\frac{q-1}{d}}= \sum_{i=1}^n c_i (a_N+a_i)^{m}=\sum_{k=0}^m \binom{m}{k} a_N^{m-k} \bigg(\sum_{i=1}^n c_ia_i^k\bigg)=0.
\end{align*} 
Thus, $a_j$ is a root of $f$ with multiplicity at least $m$ for each $1 \leq j \leq n$, and $a_N$ is in fact a root of $f$ with multiplicity at least $m+1$. It follows that
$$
\frac{(N-1)^2+1}{2}=\frac{N(N-2)}{2}+1=mN+1\leq \deg f \leq \frac{q-1}{d}.
$$
Finally, we additionally assume that $0 \in A$, that is, $a_N=0$. Then we must have $a_i \in S_d$ and thus $a_i^{\frac{q-1}{d}}=1$ for each $1 \leq i \leq n$. A similar computation shows that
\begin{align*}
E^{(r)} f (0)
&=  \sum_{i=1}^n c_i \bigg(\sum_{k=0}^m E^{(r-k)}[(x+a_i)^{m+\frac{q-1}{d}}](0) \cdot E^{(k)}[(x-a_i)^{m}](0)\bigg)\\
&=  \sum_{k=0}^m \binom{m+\frac{q-1}{d}}{r-k} \binom{m}{k} \bigg(\sum_{i=1}^n c_i a_i^{m-r+k+\frac{q-1}{d}} (-a_i)^{m-k}\bigg)\\
&=  \sum_{k=0}^m \binom{m+\frac{q-1}{d}}{r-k} \binom{m}{k} (-1)^{m-k}\bigg(\sum_{i=1}^n c_i a_i^{2m-r+\frac{q-1}{d}}\bigg)\\
&=  \sum_{k=0}^m \binom{m+\frac{q-1}{d}}{r-k} \binom{m}{k} (-1)^{m-k}\bigg(\sum_{i=1}^n c_i a_i^{2m-r}\bigg)=0
\end{align*}
for each $m+1 \leq r \leq 2m$. Thus, in this case, $0=a_N$ is in fact a root of $f$ with multiplicity at least $2m+1=n=N-1$. Consequently, we have a stronger bound that
\begin{align*}
&\frac{N(N-1)}{2}+\#\{a \in A: 2a \in S_d\}\\
&=(N-1) \cdot \frac{(N-2)}{2}+ \#\{1\leq j \leq n: 2a \in S_d\}+(N-1)\leq \deg f \leq \frac{q-1}{d}.
\end{align*}
\end{proof}

In view of Theorem~\ref{thm:ub} and Theorem~\ref{thm:VLM}, we also need to consider the case that $A\hat{+}A \subset S_d \cup \{0\}$. Next, we prove that a slightly weaker bound still holds under this weaker assumption.

\begin{prop}\label{prop: allow0}
Let $d \geq 2$ and let $q \equiv 1 \pmod d$ be an odd prime power. Let $A' \subset \F_q$ such that $A'\hat{+}A' \subset S_d \cup \{0\}$ while $A'+A' \not \subset S_d \cup \{0\}$. Then $|A'|\leq \sqrt{2(q-1)/d+1}+2$. 
\end{prop}

\begin{proof}
Since $A'\hat{+}A' \subset S_d \cup \{0\}$ and $A'+A' \not \subset S_d \cup \{0\}$, there is $a^* \in A'$ such that  $2a^* \not \subset S_d \cup \{0\}$. Note that $a^* \neq 0$ and thus $-a^* \neq a^*$. Let $A=A' \setminus \{-a^*\}$. Then $|A|=|A'|$ or $|A|=|A'|-1$. Thus, it suffices to show $|A|\leq \sqrt{2(q-1)/d+1}+1$. 

Let $A=\{a_1,a_2, \ldots, a_N\}$. We can assume that $|A|\geq 2$. Without loss of generality, we may assume that $a_1=a^*$. The proof is essentially the same as the proof of Proposition~\ref{prop:main}. We use the same notations and use the same polynomial. While equation~\eqref{eq:obs} may fail, we instead have 
$$
(a_i+a_j)^{\frac{q-1}{d}+1}(a_j-a_i)=(a_j+a_i)(a_j-a_i)
$$
for each $1 \leq i, j \leq N$. Almost identical arguments lead to the following information on $f$:
\begin{itemize}
    \item $\deg f \leq \frac{q-1}{d}$.
    \item $f(a_j)=0$ for each $1 \leq j \leq N$.
    \item a slightly weaker claim, namely
\begin{equation}
\sum_{i=1}^N c_i E^{(k_1)}[(x+a_i)^{m+\frac{q-1}{d}}](a_j) \cdot E^{(k_2)}[(x-a_i)^{m}](a_j)=0.
\end{equation}
holds for each $1 \leq j \leq N$, provided that $0 \leq k_1,k_2<m$ and $1 \leq k_1+k_2 \leq m$. Compared to the claim in the proof of Proposition~\ref{prop:main}, we also need to deal with the cases $k_1=m$ and $k_2=0$ separately. 
\item Lemma~\ref{Leibniz} then allows us to deduce that $E^{(r)} f (a_j)=0$ for $1\leq j \leq N$ and $0 \leq r \leq m-1$.
\end{itemize}

Since $-a_1=-a^* \notin A$, it follows that $a_i+a_1 \neq 0$ for each $1 \leq i \leq N$, so that we have
$$
(a_i+a_1)^{\frac{q-1}{d}}(a_1-a_i)=a_1-a_i.
$$
Thus, following a similar computation as in the deduction of equation~\eqref{Neven} and equation~\eqref{m+1} in the proof of Proposition~\ref{prop:main}, we have
\begin{align*}
E^{(m)} f (a_1)
&=  \sum_{i=1}^n c_i \bigg(\sum_{k=0}^m E^{(k)}[(x+a_i)^{m+\frac{q-1}{d}}](a_1) \cdot E^{(m-k)}[(x-a_i)^{m}](a_1)\bigg) \\
&= \sum_{i=1}^n c_i E^{(0)}[(x+a_i)^{m+\frac{q-1}{d}}](a_1) \cdot E^{(m)}[(x-a_i)^{m}](a_1)\\
&+\sum_{i=1}^n c_i E^{(m)}[(x+a_i)^{m+\frac{q-1}{d}}](a_1) \cdot E^{(0)}[(x-a_i)^{m}](a_1)\\
&= \sum_{i=1}^n c_i (a_1+a_i)^{m+\frac{q-1}{d}} +\binom{m+\frac{q-1}{d}}{m} \sum_{i=1}^{n} c_i (a_1+a_i)^{\frac{q-1}{d}}(a_1-a_i)^m\\
&= \sum_{i=1}^n c_i (a_1+a_i)^{m+\frac{q-1}{d}} +\binom{m+\frac{q-1}{d}}{m} \sum_{i=1}^{n} c_i (a_1-a_i)^m\\
&= c_1 (2a_1)^m \bigg((2a_1)^{\frac{q-1}{d}}-1\bigg) \neq 0
\end{align*}
since $2a_1 \not \in S_d \cup \{0\}$. In particular, $f$ is not identically zero. Therefore, Lemma~\ref{lem:multiplicity} implies that
$$
\frac{N(N-2)}{2}\leq Nm \leq \deg f \leq \frac{q-1}{d}.
$$
It follows that $N \leq \sqrt{2(q-1)/d+1}+1$ and the proof is complete.
\end{proof}

\subsection{Proof of Theorem~\ref{thm:ub} and Theorem~\ref{thm:VLM}, and their implications to Cayley sum graphs}\label{subsec:Cayleysum}

We first deduce Theorem~\ref{thm:ub}.

\begin{proof} [Proof of Theorem~\ref{thm:ub}]
    If $A+A \subset S_d \cup \{0\}$, then Theorem~\ref{thm:sumset} implies that $|A|^2 \leq \frac{p-1}{d}+|A|$, so $|A|\leq \sqrt{p/d}+1$. Next assume that $A +A \not \subset S_d \cup \{0\}$. Then the upper bound on $|A|$ follows from Proposition~\ref{prop:main} and Proposition~\ref{prop: allow0}.
\end{proof}

Before proving Theorem~\ref{thm:VLM}, we recall a few basic terminologies from graph theory. A {\em clique} in a graph $X$ is a subset of vertices in which every two distinct vertices are adjacent, and the {\em clique number} of $X$, denoted $\omega(X)$, is the size of a maximum clique. The graphs related to our discussions are the widely studied generalized Paley graphs.

We follow the notations in \cite{AY22, Y22}. Let $d \geq 2$ and let $q \equiv 1 \pmod {2d}$ be a prime power. The {\em $d$-Paley graph} over $\F_q$, denoted $GP(q,d)$, is the graph whose vertices are the elements of $\F_q$, where two vertices are adjacent if and only if the difference of the two vertices is a $d$-th power in $\F_q^*$. The condition $q \equiv 1 \pmod {2d}$ guarantees that $GP(q,d)$ is undirected and non-degenerate (see for example \cite[page 1]{Y22}). Now we are ready to prove Theorem~\ref{thm:VLM} with the help of existing results on $d$-Paley graphs.

\begin{proof}[Proof of Theorem~\ref{thm:VLM}]
When $q \leq 121$, we have checked the theorem by SageMath. Next, assume that $q>121$. Let $A \subset \F_q$ such that $A \hat{+} A \subset S_d \cup \{0\}$.

If $A+A \not \subset S_d \cup \{0\}$, then Proposition~\ref{prop: allow0} implies that $|A|\leq \sqrt{\frac{2(q-1)}{d}+1}+2$. Note that when $q>121$, we have
$$
|A| \leq \sqrt{\frac{2(q-1)}{d}+1}+2 \leq \sqrt{\frac{2(q-1)}{3}+1}+2<\sqrt{q}
$$
since $d \geq 3$. Next we consider the case that $A+A \subset S_d \cup \{0\}$. By Lemma~\ref{lem: trivial}, $|A|\leq \sqrt{q}$, with equality holding only if $A=-A$. Thus, we have proved the first assertion that $|A|\leq \sqrt{q}$. 

It remains to characterize $A$ such that $|A|=\sqrt{q}$ and $A+A \subset S_d \cup \{0\}$.  The above analysis shows that we must have $A=-A$ and $0 \in A$ (since $q$ is odd). Thus $A-A=A+A \subset S_d \cup \{0\}$. Since $A-A=-(A-A)$, we must have $-1 \in S_d$. Since $q \equiv 1 \pmod d$, this implies that $q \equiv 1 \pmod {2d}$, so that the $d$-Paley graph over $\F_q$, is well-defined. Using this terminology, $A$ is a clique in $GP(q,d)$ with size $\sqrt{q}$. It then follows from \cite[Theorem 1.2]{Y22} that $d \mid (\sqrt{q}+1)$. Note that the condition $d \mid (\sqrt{q}+1)$ implies that $\F_{\sqrt{q}}^* \subset S_d$.  Moreover, Sziklai \cite{Szi99} showed that if $0,1 \in A$, then $A=\F_{\sqrt{q}}$; see also \cite[Theorem 2.9]{Y22} and \cite[Section 2]{AY22}. Note that $0 \in A$ while $1$ is not necessarily in $A$, so we conclude that $A=\alpha \F_{\sqrt{q}}$ for some $\alpha \in S_d$. Conversely, it it easy to verify that, if $\alpha \in S_d$, then $\alpha \F_{\sqrt{q}} \hat{+} \alpha \F_{\sqrt{q}}=\alpha \F_{\sqrt{q}} \subset S_d \cup \{0\}$. This completes the proof of the theorem.
\end{proof}

\begin{rem}
Our analysis cannot handle the case $d=2$. In fact, there is $A \subset \F_9$ with size $4$ such that $A \hat{+} A \subset S_2 (\F_9) \cup \{0\}$, and there are $15$ different $A \subset \F_{25}$ with size $5$ such that $A \hat{+} A \subset S_2 (\F_{25}) \cup \{0\}$. Nevertheless, we conjecture that when $q$ is an odd square and $q \geq 49$, the statement of Theorem~\ref{thm:VLM} also extends to the case $d=2$. 
\end{rem}

We have seen the connection between our results and cliques in generalized Paley graphs from the above proof. Next, we further explore such a connection and reformulate our results using the language of graphs. We need to recall the notions of Cayley graphs and Cayley sum graphs.

Let $G$ be an abelian group and let $D \subset G$. The {\em Cayley graph} $\operatorname{Cay}(G, D)$  (typically one needs to further assume that $0 \not \in D$ and $D=-D$) is the undirected graph whose vertices are elements of $G$, such that two vertices $g$ and $h$ are adjacent if and only if $g-h \in D$. One can similarly define Cayley sum graphs: the {\em Cayley sum graph} $\operatorname{CayS}(G, D)$ is the undirected graph whose vertices are elements of $G$, such that two (distinct) vertices $g$ and $h$ are adjacent if and only if $g+h \in D$. In particular, note that a subset $A$ of vertices is a clique in a Cayley sum $X=\operatorname{CayS}(G, D)$ if and only if $A\hat{+} A \subset D$, thus cliques in $X$ are closely connected to restricted sumsets.  

Let $d \geq 2$ and $q \equiv 1 \pmod d$ be an odd prime power. We define the \emph{$d$-Paley sum graph} over $\F_q$, denoted $GPS(q,d)$, to be a graph with the vertex set being $\F_q$, and two vertices $x$ and $y$ are adjacent if $x+y$ is a $d$-th power in $\F_q^*$ or $x+y=0$. Equivalently, $GPS(q,d)=CayS(\F_q, S_d \cup \{0\})$. Now we are ready to equivalently reformulate Theorem~\ref{thm:ub} and Theorem~\ref{thm:VLM} using the graph theoretical language:
\begin{cor}
Let $d \geq 2$ and let $p \equiv 1 \pmod d$ be a prime. Then $$\omega(GPS(p,d))\leq \sqrt{2(p-1)/d+1}+2.$$ 
\end{cor}

\begin{cor}\label{cor:VLM}
Let $d \geq 3$. Let $q \equiv 1 \pmod d$ be an odd prime power and a square. If $d \nmid (\sqrt{q}+1)$, then $\omega(GPS(q,d))\leq \sqrt{q}-1$. If $d \mid (\sqrt{q}+1)$, then $\omega(GPS(q,d))= \sqrt{q}$ and the each maximum clique if of the form $\alpha \F_{\sqrt{q}}$, where $\alpha \in S_d$.
\end{cor}

\begin{rem}\label{rem:EKR}
The original conjecture due to van Lint and MacWilliams \cite{vLM78} is often formulated as the classification of maximum cliques in Paley graphs of square order. Its generalizations can be similarly viewed as the classification of maximum cliques in special Cayley graphs (such as generalized Paley graphs, Peisert graphs, and Peisert-type graphs) and a key ingredient for such generalizations is to view such graphs geometrically \cite{AY22, Blo84, Szi99}. The conjecture and its generalization are also known as the Erd{\H{o}}s-{K}o-{R}ado (EKR) theorem in these graphs, in the sense that each maximum clique in these graphs is a canonical clique, that is, a clique with a subfield structure, which corresponds to a line in the affine Galois plane. We refer to more discussions in \cite[Section 5.9]{GM15} and \cite{AY22}.

Corollary~\ref{cor:VLM} can be viewed as the EKR theorem for $d$-Paley sum graphs. If $d \mid (\sqrt{q}+1)$, then it is obvious that $\alpha \F_{\sqrt{q}}$ are maximum cliques (these cliques can be viewed as canonical cliques) and Corollary~\ref{cor:VLM} implies there is no maximum non-canonical clique, which is reminiscent of the classical EKR theorem \cite{EKR}. Usually, Cayley sum graphs are much more difficult to study compared to Cayley graphs. Karen Meagher \footnote{private communication} remarked that proving EKR results for Cayley sum graphs requires different tools other than those known algebraic approaches for proving EKR results for Cayley graphs collected in her book joint with Godsil~\cite{GM15} mainly because Cayley sum graphs lose vertex-transitivity. She also pointed out that Corollary~\ref{cor:VLM} appears to be the first instance of EKR results in a family of Cayley sum graphs.
\end{rem}

\section{Applications to restricted sumset decompositions}\label{sec4}

Recall that our goal is to analyze whether a multiplicative subgroup $S_d$ can admit a restricted sumset decomposition, that is, whether there is $A \subset \F_q$ such that $A \hat{+} A=S_d$. In this section, we apply Proposition~\ref{prop:main} and various additional ingredients to prove our main results. We first prove Theorem~\ref{thm: Sidon} and deduce Theorem~\ref{thm:square} in Section~\ref{sec:Sidon}. We then introduce more tools and prove Theorem Theorem~\ref{thm:main} and Theorem~\ref{thm:density}.

\subsection{The condition on sumset in Proposition~\ref{prop:main}}  To apply Proposition~\ref{prop:main}, we need to analyze the possibility of $A+A=S_d \cup \{0\}$ or $A+A=S_d$. The following proposition gives a sufficient condition for ruling out this possibility.
\begin{prop}\label{prop: A+AnotS_d}
Let $d \geq 2$ and let $p \equiv 1 \pmod d$ be a prime. Let $q$ be a power of $p$ such that $\frac{q-1}{d}\geq 3$. Then for any $A \subset \F_q$, $A+A \neq S_d$ and $A+A \neq S_d \cup \{0\}$.
\end{prop}
\begin{proof}
The statement $A+A \neq S_d$ has already been proved in \cite[Corollary 1.3]{Y24}.
Suppose otherwise that $A+A=S_d \cup \{0\}$. In this case, it is more difficult to derive a contradiction. In the following, we use a similar argument as in the proof of \cite[Theorem 1.2]{Y24} together with a few new ingredients.

For each $x \in A+A$, let $r(x)$ be the number of pairs $(a,a')$ such that $a,a' \in A$ and $a+a'=x$. Let $\beta$ be the number of $x \in (A+A)\setminus\{0\}$ such that $r(x)=1$. Observe the following:
\begin{itemize}
    \item $r(0)=|A \cap (-A)|$. If $a+a'=0$, then $a'=-a \in A$ and thus $a \in A \cap (-A)$. 
    \item $\beta \leq |A|$. Indeed, if $r(x)=1$, then there exist $a \in A$ such that $x=a+a$. Otherwise, if there are $a',a'' \in A$ such that $a' \neq a''$ and $a'+a''=x$, then $r(x) \geq 2$ since we also have $a''+a'=x$. 
\end{itemize}
It follows that
\begin{align*}
|A|^2&=\sum_{x \in A+A} r(x)\geq r(0)+\beta+2(|(A+A) \setminus \{0\}|-\beta) \\
&= |A \cap (-A)|+2|(A+A) \setminus \{0\}|-\beta \geq |A \cap (-A)|+2|(A+A) \setminus \{0\}|-|A|.
\end{align*}
Therefore, 
\begin{equation}\label{eq1}
 \frac{q-1}{d}=|S_d|=|(A+A) \setminus \{0\}|\leq \frac{|A|^2+|A|-|A \cap (-A)|}{2}.   
\end{equation}

Suppose $B$ is a subset of $A$ such that 
\begin{equation}\label{eqB}
|B|>\frac{|A|+1}{2} \quad \text{ and } \quad \binom{|B|-1+\frac{q-1}{d}}{\frac{q-1}{d}}\not \equiv 0 \pmod p.
\end{equation}
Note that $A+B \subset A+A=S_d \cup \{0\}$. Thus, Theorem~\ref{thm:sumset} implies that
\begin{equation}\label{eq2}
|A||B|\leq \frac{q-1}{d}+|A \cap (-B)|.
\end{equation}
Adding up equation~\eqref{eq1} and equation~\eqref{eq2} and simplifying, we obtain that
$$
|A|B| \leq \frac{|A|^2+|A|-|A \cap (-A)|}{2} +|A \cap (-B)|\leq \frac{|A|^2+|A|+|A \cap (-A)|}{2},
$$
with equality holding only if $|A \cap (-A)|=|A \cap (-B)|$. On the other hand, since $|B|>\frac{|A|+1}{2}$, we have
$$
|A|B|\geq \frac{|A|(|A|+2)}{2}\geq \frac{|A|^2+|A|+|A \cap (-A)|}{2}
$$
with equality holding only if $|A|=|A \cap (-A)|$, that is, $-A=A$. By comparing the above two estimates, we deduce that $A=B$ and $-A=A$. Thus equation~\eqref{eq2} implies that
$$
|A^2| \leq \frac{q-1}{d}+|A|
$$
while equation~\eqref{eq1} implies that
$$
\frac{q-1}{d} \leq \frac{|A|^2}{2},
$$
It follows that $|A| \leq 2$ and thus $|S_d|=|(A+A) \setminus \{0\}| \leq 2$, contradicting the assumption that $|S_d|=\frac{q-1}{d} \geq 3$.

It remains to construct a subset $B$ of $A$ with the property~\eqref{eqB}. Write $|A|-1=(c_k, c_{k-1}, \ldots, c_1,c_0)_p$ in base-$p$, that is, $|A|-1=\sum_{i=0}^k c_ip^i$ with $0 \leq c_i \leq p-1$ for each $0 \leq i \leq k$ and $c_k \geq 1$. Next, we construct $B$ according to the size of $c_k$.

(1) $c_k \leq p-1-\frac{p-1}{d}$. In this case, let $B$ be an arbitrary subset of $A$ with $|B|-1=(c_k,0, \ldots, 0)_p$, that is, $|B|=c_kp^k+1$. It is easy to verify that $\binom{|B|-1+\frac{q-1}{d}}{\frac{q-1}{d}} \not \equiv 0 \pmod p$ using Kummer's theorem. Since $|A| \leq (c_k+1)p^k$, we also have that $2|B|>|A|+1$.

(2) $c_k > p-1-\frac{p-1}{d}$. In this case, let $B$ be an arbitrary subset of $A$ with $$|B|-1=\bigg(\frac{(d-1)(p-1)}{d},\frac{(d-1)(p-1)}{d}, \ldots, \frac{(d-1)(p-1)}{d}\bigg)_p,$$ that is, $|B|=\frac{(d-1)(p-1)}{d} \cdot \sum_{i=0}^k p^i +1$. Again, it is easy to verify that $\binom{|B|-1+\frac{q-1}{d}}{\frac{q-1}{d}} \not \equiv 0 \pmod p$. Since $d \geq 2$, it follows that $2|B| \geq (p-1)\sum_{i=0}^k p^i +2= p^{k+1}+1\geq |A|+1$, where equality holds only if $d=2$ and $|A|=p^{k+1}$. 

(3) It remains to consider the case that $d=2$ and $|A|=p^{k+1}$. Equation~\eqref{eq1} implies that $q-1 \leq |A|^2+|A|-|A \cap (-A)|$ and thus $|A|\geq \sqrt{q}-1$. On the other hand,  $|A|\leq \sqrt{q}$, where equality holds only if $A=-A$. Therefore, $q$ is a square, $|A|=\sqrt{q}$ and $A=-A$. Since $q$ is odd, $0 \in A$ and $A \subset S_2 \cup \{0\}$. It follows that $A-A=A+A=S_2 \cup \{0\}$. Let $A'=a_0^{-1}A$, where $a_0$ is a nonzero element of $A$. Then we have $0,1 \in A'$, $|A'|=\sqrt{q}$ and $A'-A' \subset S_2 \cup \{0\}$. Now the conjecture of van Lint and MacWilliams \cite{vLM78}  mentioned in the introduction (first confirmed by Blokhuis \cite{Blo84}) implies that $A'=\F_{\sqrt{q}}$ and thus $A=a_0\F_{\sqrt{q}}$. However, this implies that $A+A=a_0\F_{\sqrt{q}} \neq S_2 \cup \{0\}$, a contradiction.  
\end{proof}

\subsection{Proof of Theorem~\ref{thm: Sidon} and its consequences}\label{sec:Sidon}

Proposition~\ref{prop:main} implies the following corollary on strong structural information on $A$, which will be crucial in the proof of our main results.

\begin{cor}\label{cor:Sidon}
Let $d \geq 2$ and let $q \equiv 1 \pmod d$ be an odd prime power. Assume that there is $A \subset \F_q$ such that $A\hat{+}A=S_d$ while $A+A \not \subset S_d \cup \{0\}$. If $|A|$ is odd or $0 \in A$, then $A$ is a Sidon set with $\{2a: a \in A\} \cap S_d=\emptyset$ and 
$$q=\frac{d|A|(|A|-1)}{2}+1.$$  
If $|A|$ is even, then 
$$
N=2\bigg\lceil \sqrt{\frac{q-1}{2d}} \bigg \rceil, \quad \text{and} \quad \sqrt{\frac{q-1}{2d}} \in \bigg(\frac{1}{2},\frac{3}{4}\bigg) \pmod 1.
$$
\end{cor}
\begin{proof}
Let $|A|=\{a_1,a_2, \ldots, a_N\}$. Note that $A \hat{+} A=S_d$ implies that
\begin{equation}\label{eq:bound}
\frac{q-1}{d}=|S_d| \leq |A \hat{+} A|\leq \binom{N}{2}=\frac{N(N-1)}{2}.    
\end{equation}
Next, we divide the discussion according to the following two cases.

(1) $N$ is odd. Then Proposition~\ref{prop:main} implies that
$$
\frac{N(N-1)}{2}+\#\{a \in A: 2a \in S_d \cup \{0\}\} \leq \frac{q-1}{d}.
$$
It follows from equation~\eqref{eq:bound} that
$$
N(N-1)=\frac{2(q-1)}{d},
$$
that is, $A$ is a weak Sidon set (the sums $a_i+a_j$, for $i<j$, are all distinct). Moreover, $2a \not \in S_d \cup \{0\}$ for each $a \in A$. Since $q$ is odd, it is clear that $2a_i \neq 2a_j$ for $i \neq j$. Thus, we conclude that $A$ is a Sidon set.

(2) $N$ is even and $0 \in A$. Then Proposition~\ref{prop:main} implies that
$$\frac{N(N-1)}{2}+\#\{a \in A: 2a \in S_d\}  \leq \frac{q-1}{d}.$$
We can argue similarly to (1) to deduce the desired properties of $A$.

(3) $N$ is even and $0 \not \in A$. Then we have the weaker estimate
$$
\frac{(N-1)^2+1}{2} \leq \frac{q-1}{d}
$$
from Proposition~\ref{prop:main}. It follows from equation~\eqref{eq:bound} that
$$
(N-1)^2< \frac{2(q-1)}{d} \leq N(N-1)<(N-1/2)^2,
$$
equivalently,
$$
\frac{N}{2}-\frac{1}{2}< \sqrt{\frac{q-1}{2d}}<\frac{N}{2}-\frac{1}{4}.
$$
Recall that $N$ is even, so $\frac{N}{2}$ is an integer. Thus, $N$ is uniquely determined by 
\begin{equation}
N=2\bigg\lceil \sqrt{\frac{q-1}{2d}} \bigg \rceil, \quad \text{and} \quad \sqrt{\frac{q-1}{2d}} \in \bigg(\frac{1}{2},\frac{3}{4}\bigg) \pmod 1.
\end{equation}
\end{proof}

Next, we present the proof of Theorem~\ref{thm: Sidon}.
\begin{proof}[Proof of Theorem~\ref{thm: Sidon}]
Let $A \subset \F_q$ such that $A\hat{+}A=S_d$. In view of Corollary~\ref{cor:Sidon}, it suffices to show that $A+A \not \subset S_d \cup \{0\}$. Equivalently, it suffices to show that $A+A \neq S_d$ and $A+A \neq S_d \cup \{0\}$, which have been proved in Proposition~\ref{prop: A+AnotS_d}.
\end{proof}

With Theorem~\ref{thm: Sidon}, we deduce two interesting consequences below.

\begin{proof}[Proof of Theorem~\ref{thm:square}]
Let $q=p^{2r}$. Assume that there is $A \subset \F_q$ with $A \hat{+} A= S_d$. 

(1) If $|A|$ is odd, then Theorem~\ref{thm: Sidon} implies that $q=p^{2r}=k^2 |A|(|A|-1)+1.$ Thus,
$$
(2p^r)^2=4q=4k^2|A|(|A|-1)+1=(2k|A|-k)^2+(4-k^2).
$$
It follows that $(2p^r+2k|A|-k)$ is a divisor of $(k^2-4)$. In particular, $p \leq 2p^r+2k|A|-k \leq k^2-4<d$, violating the assumption that $p \equiv 1 \pmod d$.

(2) If $|A|$ is even, then Theorem~\ref{thm: Sidon} implies that 
$$
\frac{p^r-1}{2k}+\frac{\sqrt{p^{2r}-1}-(p^r-1)}{2k}=\frac{\sqrt{p^{2r}-1}}{2k}=\sqrt{\frac{q-1}{2d}} \in \bigg(\frac{1}{2},\frac{3}{4}\bigg) \pmod 1.
$$
Since $p \equiv 1 \pmod d$ and $d=2k^2$, it follows that $2k \mid (p^r-1)$ and thus
$$
\frac{\sqrt{p^{2r}-1}-(p^r-1)}{2k}\in \bigg(\frac{1}{2},\frac{3}{4}\bigg) \pmod 1.
$$
However, the number on the left-hand side is in $(0,\frac{1}{2k})$, while $\frac{1}{2k}\leq \frac{1}{6}<\frac{1}{2}$, a contradiction.
\end{proof}

\begin{proof} [Proof of Corollary~\ref{cor:Siegel}]
Let $p \equiv 1 \pmod d$ to be a prime and let $A \subset \F_{p^k}$ such that $A\hat{+}A=S_d(\F_{p^k})$. Assume that $|A|$ is odd or $0 \in A$. Then Theorem~\ref{thm: Sidon} implies that
$$p^k=\frac{d|A|(|A|-1)}{2}+1.$$  

First, consider the case $d=8$. Then we have $p^k=4|A|(|A|-1)+1=(2|A|-1)^2$, which is impossible since $k$ is odd.  

Next, we assume that $d \neq 8$. Note that
$$
C: y^2=f(x)=\frac{8}{d}(x^k-1)+1
$$
is a smooth affine curve with genus $g=\lceil k/2 \rceil-1 \geq 1$ since $f(x)=\frac{8}{d}(x^k-1)+1$ has no repeated root when $d \neq 8$ (see for example \cite[Section A.4.5]{HS00}). Thus, Siegel's theorem on integral points \cite[Theorem D.9.1]{HS00} implies the curve $C$ only has finitely many integral points. On the other hand, note that $(p, 2|A|-1)$ is an integral point on the curve $C$, thus $p$ is upper bounded by a constant $p_0$ depending only on $d$ and $k$.
\end{proof}

\subsection{A refinement on Proposition~\ref{prop:main} when $|A|$ is even.}

When $|A|$ is even and $0 \notin A$, Proposition~\ref{prop:main} is not strong enough for our purpose. Thus, we are led to pay extra attention to this more challenging situation and we establish the following proposition. 

\begin{prop}\label{prop: even}
Let $d \geq 2$ and let $q \equiv 1 \pmod d$ be an odd prime power. Assume that there is $A \subset \F_q$ with $|A|=N$ even, such that $A\hat{+}A=S_d$ while $A+A \not \subset S_d \cup \{0\}$. Then the following conditions necessarily hold:
\begin{itemize}
    \item $\binom{m+\frac{q-1}{d}}{m+1} \not \equiv 0 \pmod p$, where $m=\frac{N-2}{2}$.
     \item $\{2a: a \in A\} \cap S_d=\emptyset$.
     \item If $d=2$, then $q \equiv 3,5 \pmod 8$ and $A$ is a Sidon set with $0 \in A$.
\end{itemize}
\end{prop}

\begin{proof}
Let $A=\{a_1, a_2, \ldots, a_N\}$ and let $m=\frac{N-2}{2}$. Consider instead the following auxiliary polynomial 
\begin{equation}\label{poly2}
f(x)=-(-1)^{m+1}+\sum_{i=1}^N c_i (x+a_i)^{m+\frac{q-1}{d}} (x-a_i)^{m+1}\in \F_q[x],    
\end{equation}
where $c_1,c_2,...,c_N$ is the unique solution of the following system of equations:
\begin{equation} \label{system2} 
\left\{
\TABbinary\tabbedCenterstack[l]{
\sum_{i=1}^N c_i a_i^j=0,  \quad 0 \leq j \leq 2m=N-2\\\\
\sum_{i=1}^N c_i a_i^{N-1}=1.
}\right.    
\end{equation}

With this new polynomial~\eqref{poly2}, similar arguments as in the proof of Proposition~\ref{prop:main} lead to the following:
\begin{itemize}
    \item $\deg f \leq \frac{q-1}{d}$;
    \item $f(a_j)=0$ for each $1 \leq j \leq N$;
    \item an analogous claim, namely
\begin{equation}\label{eq:claim2}
\sum_{i=1}^N c_i E^{(k_1)}[(x+a_i)^{m+\frac{q-1}{d}}](a_j) \cdot E^{(k_2)}[(x-a_i)^{m+1}](a_j)=0.
\end{equation}
holds for each $1 \leq j \leq N$, provided that $0 \leq k_1,k_2\leq m$ and $1 \leq k_1+k_2 \leq m+1$.
\item Via claim~\eqref{eq:claim2}, Lemma~\ref{Leibniz} then allows us to deduce that for each $1 \leq j \leq N$, the multiplicity of $a_j$ as a root of $f$ is at least $m+1$ by showing that $E^{(r)}f(a_j)=0$ for $1 \leq r \leq m$. 
\end{itemize}

If $f$ is not identically zero, then we obtain the inequality
$$
\frac{N^2}{2}=N(m+1)\leq \deg f \leq \frac{q-1}{d},
$$
which contradicts inequality~\eqref{eq:bound} (a consequence of $A \hat{+} A=S_d$). Thus, we deduce that the polynomial $f$ we constructed in equation~\eqref{poly2} must be identically zero. Next, we use this strong algebraic condition to deduce some structural information of $A$.

Since $f$ is identically zero, in particular, $E^{(m+1)} f (a_j)=0$ for each $1 \leq j \leq N$. For each $1\leq j \leq N$, note that claim~\eqref{eq:claim2} and Lemma~\ref{Leibniz} together imply that
\begin{align}
E^{(m+1)} f (a_j)
&= \sum_{i=1}^N c_i E^{(0)}[(x+a_i)^{m+\frac{q-1}{d}}](a_j) \cdot E^{(m+1)}[(x-a_i)^{m+1}](a_j) \notag\\
&+\binom{m+\frac{q-1}{d}}{m+1} \sum_{i=1}^N c_i E^{(m+1)}[(x+a_i)^{m+\frac{q-1}{d}}](a_j) \cdot E^{(0)}[(x-a_i)^{m+1}](a_j) \notag\\
&= \sum_{i=1}^N c_i (a_j+a_i)^{m+\frac{q-1}{d}} +\binom{m+\frac{q-1}{d}}{m+1} \sum_{i=1}^N c_i (a_j+a_i)^{\frac{q-1}{d}-1} (a_j-a_i)^{m+1}. \label{m++}
\end{align}
In the following discussion, let $1 \leq j \leq N$ be such that $a_j \neq 0$. The same computation as in equation~\eqref{m+1} shows that
$$
\sum_{i=1}^N c_i (a_j+a_i)^{m+\frac{q-1}{d}}=c_j (2a_j)^m \bigg((2a_j)^{\frac{q-1}{d}}-1\bigg).
$$
Also note that for each $i \neq j$, 
$$
(a_j+a_i)^{\frac{q-1}{d}-1} (a_j-a_i)^{m+1}=\frac{(a_j-a_i)^{m+1}}{a_j+a_i},
$$
and the above equation holds even when $i=j$ and $a_j=0$, since both sides are zero. In view of system~\eqref{system2}, we have the following: 
\begin{align*}
\sum_{i=1}^N c_i \frac{(a_j-a_i)^{m+1}}{a_j+a_i}
&=\sum_{i=1}^N c_i \frac{(2a_j-(a_j+a_i))^{m+1}}{a_j+a_i}\\
&= \sum_{k=0}^{m+1} \binom{m+1}{k} (2a_j)^k (-1)^{m+1-k} \bigg(\sum_{i=1}^N c_i (a_j+a_i)^{m-k}\bigg)\\
&=(2a_j)^{m+1} \sum_{i=1}^N \frac{c_i}{a_j+a_i}.
\end{align*}
Therefore, using the above three equations, equation~\eqref{m++} simplifies to
\begin{equation}\label{0=}
0=E^{(m+1)} f (a_j)=c_j (2a_j)^m \bigg((2a_j)^{\frac{q-1}{d}}-1\bigg) +\binom{m+\frac{q-1}{d}}{m+1} (2a_j)^{m+1} \sum_{i=1}^N \frac{c_i}{a_j+a_i}.
\end{equation}

Using the formula for the inverse of a Vandermonde matrix (see for example \cite[Section 0.9.11]{HJ13}), it is easy to verify that
$$
c_i=\bigg(\prod_{\substack{1 \leq k \leq N\\k \neq i}} (a_i-a_k)\bigg)^{-1}.
$$

Consider the polynomial
$$
g(x)=\sum_{i=1}^N c_i\prod_{k \neq i} (x+a_k)=\sum_{i=1}^N \frac{\prod_{k \neq i} (x+a_k)}{\prod_{k \neq i} (a_i-a_k)} \in \F_q[x].
$$
Note that $g$ has degree at most $N-1$, and for each $1 \leq i \leq N$, we have $g(-a_i)=(-1)^{N-1}=-1$. Therefore, $g \equiv -1$ and thus
$$
\sum_{i=1}^N \frac{c_i}{a_j+a_i}=\frac{g(a_j)}{\prod_{i=1}^N (a_j+a_i)}=\frac{-1}{\prod_{i=1}^N (a_j+a_i)}.
$$

As a summary, if $a_j \neq 0$, we have shown that
\begin{align*}
c_j (2a_j)^m \bigg((2a_j)^{\frac{q-1}{d}}-1\bigg) =\binom{m+\frac{q-1}{d}}{m+1}  (2a_j)^{m+1} \frac{1}{\prod_{i=1}^N (a_j+a_i)},
\end{align*}
equivalently
\begin{equation}\label{eq:nonzero}
c_j \bigg((2a_j)^{\frac{q-1}{d}}-1\bigg) \cdot \prod_{i=1}^N (a_j+a_i)= \binom{m+\frac{q-1}{d}}{m+1} \cdot 2a_j.
\end{equation}

Since $A+A \not \subset S_d \cup \{0\}$, there is $1\leq j_0\leq N$ such that $2a_{j_0} \not \in S_d \cup \{0\}$. By setting $j=j_0$ in equation~\eqref{eq:nonzero}, it follows that $\binom{m+\frac{q-1}{d}}{m+1}\neq 0$. Now if $a_j \neq 0$, equation~\eqref{eq:nonzero} implies that $(2a_j)^{\frac{q-1}{d}}\neq 1$, that is, $2a_j \notin S_d$. Therefore, $\{2a: a \in A\} \cap S_d=\emptyset$.

Finally, assume that $d=2$. We have
\begin{align*}
0=f(0)
&= -(-1)^{m+1}+\sum_{i=1}^N a_i^{m+\frac{q-1}{2}} (-a_i)^{m+1}\\
&= -(-1)^{m+1}+ (-1)^{m+1} \cdot \sum_{i=1}^N c_i a_i^{N-1+\frac{q-1}{2}}.
\end{align*}
In view of the equation $\sum_{i=1}^N c_i a_i^{N-1}=1$ in system~\eqref{system2}, we have
$$
\sum_{i=1}^N c_i a_i^{N-1+\frac{q-1}{2}}=\sum_{i=1}^N c_i a_i^{N-1}=1.
$$
Therefore, $\sum^* c_i a_i^{N-1}=1$, where the sum is over those $i$ such that $a_i \in S_2$. In particular, there is $1\leq k\leq N$ such that $a_k \in S_2$. If $q \equiv \pm 1 \pmod 8$, then $2 \in S_2$ and thus $2a_k \in S_2$, contradicting the previous necessary condition. Thus, we must have $q \equiv 3,5 \pmod 8$ so that $2 \not \in S_2$. It follows that $A \subset S_2 \cup \{0\}$ since $2 \not \in S_2$ and $\{2a: a \in A\} \cap S_2=\emptyset$. 

Assume that $0 \not \in A$ so that $A \subset S_2$. Set $A'=A \cup \{0\}$. We still have $A' \hat{+}A'=S_2$ since $A \subset S_2$. Since $|A|$ is even, it follows that $|A'|$ is odd and $A'$ is a Sidon set by Corollary~\ref{cor:Sidon}. In particular, for each $a \in A$, we have $a=0+a \not \in A+A$ while $a \in S_2$, which implies that $A\hat{+}A\neq S_2$, a contradiction. Therefore, we must have $0 \in A$. Consequently, Corollary~\ref{cor:Sidon} implies that $A$ is a Sidon set. 
\end{proof}

\subsection{Proof of Theorem~\ref{thm:main}}

\begin{proof}[Proof of Theorem~\ref{thm:main}]
Let $q=p^r>13$ and let $A \subset \F_q$. Suppose otherwise that $A \hat{+} A=S_2$. Let $|A|=N$. By Proposition~\ref{prop: A+AnotS_d}, Corollary~\ref{cor:Sidon} and Proposition~\ref{prop: even}, we deduce that $A$ is a Sidon set and $q=p^r=N^2-N+1$. If $r=1$, then $q=p$, and the possibility that $A \hat{+} A =S_2$ has been ruled out by Shkredov \cite[Theorem 1.2]{S14}. 

Next, we consider the case $r \geq 2$. It turns out that $p^r=N(N-1)+1$ (where $r \geq 2$) is a special case of the well-studied Nagell-Ljunggren equation. A classical result of Nagell \cite{N1920} (see also the survey \cite{BM02} by Bugeaud and Mignotte) implies that the only solution to $$p^r=N(N-1)+1=\frac{(N-1)^3-1}{N-1}$$ where $r \geq 2$ is $(p,r, N)=(7,3,19)$. Thus, $q=343$ and $|A|=N=19$. However, in this case, a simple code via SageMath indicates that such $A$ does not exist; in fact, the largest $B \subset \F_{343}$ such that $B \hat{+}B \subset S_2$ has size $10$. This completes the proof.
\end{proof}

\begin{rem}
When $q \in \{3,7,13\}$, as already shown by Shkredov \cite[Theorem 1.2]{S14}, there is $A \subset \F_q$ such that $A \hat{+} A=S_2$. Thus, from the above proof, we can further deduce that if $q$ is an odd prime power and $A \subset \F_q$ such that $A\hat{+}A= S_2$, then the only possibilities are the following: 
\begin{itemize}
    \item $q=3, A=\{0,1\}$.
    \item $q=7, A=\{3,5,6\}$.
    \item $q=13$, $A=\{0,1,3,9\}$ and $A=\{0,4,10,12\}$.
\end{itemize}
Note that in each of the above cases, $A$ is a Sidon set, and $q=|A|^2-|A|+1$.
\end{rem}

\subsection{Proof of Theorem~\ref{thm:density}}
Let $\mathcal{P}$ be the set of primes. For $d \geq 2$, we define $\mathcal{P}_d=\{p \in \mathcal{P}: p \equiv 1 \pmod d\}$. To prove Theorem~\ref{thm:density}, we need the following result on uniform distribution due to Bergelson et.~al. \cite[Corollary 2.3]{BKMST14}.

\begin{lem}[Bergelson et.~al. \footnote{There are some inaccuracies in the statement of \cite[Corollary 2.3]{BKMST14}. A corrected version can be found in \cite[Corollary 6.3]{Y21}.}]\label{lem:ud}
Let $0 <\theta_1 <\theta_2 < \cdots< \theta_\ell$ and let $\gamma_1, \gamma_2, \ldots, \gamma_\ell$ be nonzero real numbers such that  $\gamma_j \not \in \mathbb{Q}$ if $\theta_j \in \N$. Then for any $h \in \Z$ and positive integer $d$,  the sequence $$\bigg(\big(\gamma_1 (p-h)^{\theta_1},\gamma_2 (p-h)^{\theta_2}, \ldots, \gamma_\ell (p-h)^{\theta_\ell}\big)\bigg)_{p \in \mathcal{P}_d}$$ is equidistributed modulo $1$ in the $\ell$-dimensional torus $\mathbb{T}^\ell=\R^\ell/\Z^\ell$.
\end{lem}

Now we use the tools developed in this section to establish Theorem~\ref{thm:density}.

\begin{proof}[Proof of Theorem~\ref{thm:density}]
Let 
$$
\mathcal{B}_d=\{p \in \mathcal{P}_d: \text{there is some } A \subset \F_q \text{ with } A\hat{+}A=S_d(\F_q).\}
$$
Our aim is to show the relative upper density of $\mathcal{B}_d \subset \mathcal{P}_d$ is at most $\frac{1}{4} \cdot (\frac{d-1}{d})^r$. 

We first consider the case where $s$ is odd. In this case, $s=2r+1$. For simplicity, for each $p \in \mathcal{P}_d$, we write $q=p^{2r+1}$ and
$$
\alpha_p=\bigg\lceil \sqrt{\frac{q-1}{2d}} \bigg \rceil. 
$$
Let $p \in \mathcal{B}_d$. Then there is $A \subset \F_q$ with $A\hat{+}A=S_d$. By Proposition~\ref{prop: A+AnotS_d}, $A+A \not \subset S_d \cup \{0\}$. Let $|A|=N$. 

(1) If $N$ is odd, then Theorem~\ref{thm: Sidon} implies that
$$
q=\frac{dN(N-1)}{2}+1.
$$

(2) If $N$ is even, then Theorem~\ref{thm: Sidon} implies that $N=2\alpha_p$ and
\begin{equation}
\sqrt{\frac{q-1}{2d}} \in \bigg(\frac{1}{2},\frac{3}{4}\bigg) \pmod 1.
\end{equation}

Let $m=\frac{N-2}{2}$. By Proposition~\ref{prop: A+AnotS_d}, $A+A \not \subset S_d \cup \{0\}$ so that we can apply Proposition~\ref{prop: even} to deduce that
$$
\binom{m+\frac{q-1}{d}}{m+1} \not \equiv 0 \pmod p.
$$
Note that $\alpha_p \in (p^r, p^{r+1})$. Write $\alpha_p=m+1=(c_r, c_{r-1}, \ldots, c_0)_p$ in its base-$p$ representation. Then Kummer's theorem implies that $c_j \leq \frac{(d-1)(p-1)}{d}$ for $j \geq 1$ and $c_0 \leq \frac{(d-1)(p-1)}{d}+1$. Note that $c_j=\lfloor p\{\alpha_p/p^{j+1}\} \rfloor$ for $0 \leq j \leq r$, where $\{x\}$ denotes the fractional part of a real number $x$. Indeed, in view of the base-$p$ representation of $\alpha_p$, we have $$c_jp^j+\sum_{i=j+1}^r c_ip^i\leq \alpha_p <(c_j+1)p^j+\sum_{i=j+1}^r c_ip^i,$$ thus $c_j/p\leq \{\alpha_p/p^{j+1}\}<(c_j+1)/p$, equivalently, 
$c_j=\lfloor p\{\alpha_p/p^{j+1}\} \rfloor$.

The above discussion shows that $\mathcal{B}_d \subset \mathcal{C}_d \cup \mathcal{D}_d$, where
$$
\mathcal{C}_d= \bigg\{p \in \mathcal{P}_d: p^{2r+1}=\frac{dk(k-1)}{2}+1 \text{ for some positive integer } k\bigg\}
$$
and
\begin{multline*}
\mathcal{D}_d= \bigg\{p \in \mathcal{P}_d: \sqrt{\frac{q-1}{2d}} \in \bigg[\frac{1}{2},\frac{3}{4}\bigg) \pmod 1\bigg\} \bigcap \\ \bigcap_{j=1}^{r}  \bigg\{p \in \mathcal{P}_d: \frac{\alpha_p}{p^{j}} \in \bigg [0,\frac{(d-1)(p-1)}{dp}+\frac{2}{p}\bigg) \pmod 1  \bigg\}.
\end{multline*}

In view of the prime number theorem for arithmetic progressions and Corollary~\ref{cor:Siegel}, it is clear that the relative density of $\mathcal{C}_d \subset \mathcal{P}_d$ is $0$. Note that as $p \to \infty$,
\begin{equation}\label{eq:asymptotic}
\alpha_p=\sqrt{\frac{q}{2d}}+o(1)=\frac{p^{r+1/2}}{\sqrt{2d}}+o(1), \quad \frac{(d-1)(p-1)}{dp}+\frac{2}{p}=\frac{d-1}{d}+o(1).    
\end{equation}
Thus, in view of Lemma~\ref{lem:ud}, it is easy to verify that the relative upper density of $\mathcal{B}_d \subset \mathcal{P}_d$ is at most the relative upper density of $\widetilde{\mathcal{D}_d} \subset \mathcal{P}_d$, where $\widetilde{\mathcal{D}_d}$ is essentially obtained by dropping the $o(1)$ error terms from equation~\eqref{eq:asymptotic}. More precisely, 
\begin{multline*}
\widetilde{\mathcal{D}_d}=\bigg\{p \in \mathcal{P}_d: \frac{p^{r+1/2}}{\sqrt{2d}} \in \bigg(\frac{1}{2},\frac{3}{4}\bigg) \pmod 1, \\ \frac{p^{j+1/2}}{\sqrt{2d}} \in \bigg(0,\frac{d-1}{d}\bigg) \pmod 1 \text{ for each } 0 \leq j \leq r-1 \bigg\}.
\end{multline*}
Lemma~\ref{lem:ud} then implies that the relative density of $\widetilde{\mathcal{D}_d} \subset \mathcal{P}_d$ is $\frac{1}{4} \cdot (\frac{d-1}{d})^r$, as desired.

Finally, we consider the case where $s$ is even. In this case, $s=2r+2$. Using an almost identical argument, we can show that the relative upper density of $\mathcal{B}_d \subset \mathcal{P}_d$ is at most the relative upper density of $\widetilde{\mathcal{D}_d} \subset \mathcal{P}_d$, where
\begin{multline*}
\widetilde{\mathcal{D}_d}=\bigg\{p \in \mathcal{P}_d: \frac{p^{r+1}}{\sqrt{2d}} \in \bigg(\frac{1}{2},\frac{3}{4}\bigg) \pmod 1, \\ \frac{p^{j}}{\sqrt{2d}} \in \bigg(0,\frac{d-1}{d}\bigg) \pmod 1 \text{ for each } 1 \leq j \leq r \bigg\}.
\end{multline*}
Since $2d$ is not a perfect square, it follows that $\sqrt{2d}$ is not rational, and thus Lemma~\ref{lem:ud} implies that the relative density of $\widetilde{\mathcal{D}_d} \subset \mathcal{P}_d$ is $\frac{1}{4} \cdot (\frac{d-1}{d})^r$, as desired.
\end{proof}

\section{Applications to the question of Erd\H{o}s and Moser}\label{sec5}

First, we recall Gallagher's larger sieve \cite{G71}.
\begin{lem}\label{GS}  
Let $N$ be a natural number and  $A\subset\{1,2,\ldots, N\}$.  Let ${\mathcal P}$ be a set of primes. For each prime $p \in {\mathcal P}$, let $A_p=A \pmod{p}$. For any $1<Q\leq N$, we have
$$
 |A|\leq \frac{\underset{p\leq Q, p\in \mathcal{P}}\sum\log p - \log N}{\underset{p\leq Q, p \in \mathcal{P}}\sum\frac{\log p}{|A_p|}-\log N},
$$
provided that the denominator is positive.
\end{lem}

We conclude the paper with the proof of Theorem~\ref{thm:integers}.

\begin{proof}[Proof of Theorem~\ref{thm:integers}]
Let $A\subset\{1,2,\ldots, N\}$ such that $A \hat{+} A$ is contained in the set of $d$-th powers. To apply the Gallagher sieve inequality, we consider the set of primes 
$\mathcal{P}=\{p: p \equiv 1 \pmod d\}.$ For each prime $p \in \mathcal{P}$, denote by $A_{p}$ the image of $A \pmod{p}$. For each $p \in \mathcal{P}$, we can naturally view $A_p$ as a subset of $\F_p$ so that $A_p \hat{+} A_p \subset S_d(\F_p) \cup \{0\}$, and Theorem~\ref{thm:ub} implies that 
\begin{equation}\label{Ap}
|A_{p}| \leq \sqrt{2(p-1)/d+1}+2.    
\end{equation}
Set $Q=\frac{2}{d}(\phi(d)\log N)^2$. By the prime number theorem and standard partial summation, we have
\begin{equation}\label{prime}    
\sum_{p \in \mathcal{P},p \le Q}\log p \sim \frac{Q}{\phi(d)}, \quad \sum_{p \in \mathcal{P},p \le Q} \frac{\log p}{\sqrt{p}} \sim \frac{2\sqrt{Q}}{\phi(d)}.
\end{equation}
Thus, applying Lemma~\ref{GS} and combining the above estimates~\eqref{Ap} and \eqref{prime}, we conclude that 
\begin{align*}
|A| 
&\leq \frac{\sum_{p \in \mathcal{P},p \le Q}\log p - \log N}{\sum_{p \in \mathcal{P},p \le Q} \frac{\log p}{|A_{p}|}-\log N}
\leq \frac{\frac{(1+o(1))Q}{\phi(d)}- \log N}{(1+o(1))\sqrt{\frac{d}{2}}\cdot \frac{2\sqrt{Q}}{\phi(d)}-\log N}\\
&\leq \frac{\frac{2+o(1)}{d}\phi(d)(\log N)^2}{(2+o(1))\log N-\log N}=\frac{(2+o(1))\phi(d)}{d}\log N.
\end{align*}
\end{proof}

\section*{Acknowledgements}
The author thanks Shamil Asgarli, Ritesh Goenka, Seoyoung Kim, Shuxing Li, Greg Martin, Karen Meagher, Venkata Raghu Tej Pantangi, Misha Rudnev, Ilya Shkredov, J\'ozsef Solymosi, Zixiang Xu, and Semin Yoo for inspiring discussions.  The author is also grateful to anonymous referees for their valuable comments and suggestions. The research of the author was supported in part by an NSERC fellowship.

\bibliographystyle{abbrv}
\bibliography{main}

\end{document}